\documentclass[11pt,reqno]{amsart}

\usepackage[a4paper,margin=1in]{geometry}
\usepackage{amsmath,amssymb,amsthm,mathtools}
\usepackage{enumitem}
\usepackage{booktabs}
\usepackage{xcolor}
\usepackage[colorlinks=true,
            linkcolor=blue!60!black,
            citecolor=blue!60!black,
            urlcolor=blue!60!black]{hyperref}

\newtheorem{theorem}{Theorem}[section]
\newtheorem{proposition}[theorem]{Proposition}
\newtheorem{lemma}[theorem]{Lemma}
\newtheorem{corollary}[theorem]{Corollary}

\theoremstyle{definition}

\newtheorem{example}[theorem]{Example}

\theoremstyle{remark}
\newtheorem{remark}[theorem]{Remark}

\DeclareMathOperator{\Ann}{Ann}
\DeclareMathOperator{\Z}{Z}
\DeclareMathOperator{\U}{U}
\newcommand{\WG}{W\Gamma}
\newcommand{\m}{\mathfrak{m}}
\newcommand{\F}{\mathbb{F}}

\title[Laplacian Spectrum of the Weakly Zero-Divisor Graph]%
{Laplacian Spectrum of the Weakly Zero-Divisor Graph\\
of a Finite Commutative Ring}

\author{Hampher Shylla}
\address{Department of Mathematics, North Eastern Hill University,
Shillong 793022, India}
\email{xyza3956@gmail.com}

\author{Sainkupar Mn Mawiong}
\address{Department of Basic Sciences and Social Sciences, North
Eastern Hill University, Shillong 793022, India}
\email{skupar@gmail.com}

\author{John Paul Jala Kharbhih}
\address{Department of Mathematics, North Eastern Hill University,
Shillong 793022, India}
\email{jpkharbhih@gmail.com}



\date{\today}

\begin{document}
\maketitle

\begin{abstract}
For a commutative ring $R$ with identity, the \emph{weakly zero-divisor
graph} $\WG(R)$ has vertex set $\Z(R)^{\ast}$, with distinct vertices
$x$ and $y$ adjacent whenever there exist nonzero $r\in\Ann(x)$ and
$s\in\Ann(y)$ with $rs=0$. The Laplacian spectrum of $\WG(\Z_n)$ has
been determined by Shariq, Mathil, and Kumar, who also established that
$\WG(\Z_n)$ is Laplacian integral. Building on the structural
description of $\WG(R)$ due to Nikmehr, Azadi, and Nikandish, we extend
the Laplacian spectrum and integrality results from $\Z_n$ to
\emph{every} finite commutative ring $R$: we restate
$\WG(R)$ in unified form as a complete multipartite graph whose parts
are made explicit by the local-ring decomposition of $R$, compute the
full Laplacian spectrum in closed form, prove Laplacian integrality of
$\WG(R)$, and give a sharp bound on the number of distinct Laplacian
eigenvalues. As consequences we obtain explicit formulas for the
algebraic connectivity and number of spanning trees of $\WG(R)$, and
recover the Laplacian spectrum of $\WG(\Z_n)$ in compact form.
\end{abstract}

\section{Introduction}

The interplay between the algebraic structure of a commutative ring and
the combinatorial structure of an associated graph has been studied
intensively since Beck's seminal work \cite{Beck} on zero-divisor
graphs. Anderson and Livingston \cite{AL} refined Beck's construction
and defined the \emph{zero-divisor graph} $\Gamma(R)$ on the nonzero
zero-divisors of a commutative ring $R$, with $x \sim y$ iff $xy=0$.
The graph $\Gamma(R)$ has been studied extensively from many
perspectives; see \cite{AL} and the references therein.

Several variants have been proposed to capture different aspects of the
zero-divisor structure. Among the most natural is the
\emph{weakly zero-divisor graph} $\WG(R)$, introduced by Nikmehr, Azadi,
and Nikandish \cite{NNB}: two distinct nonzero zero-divisors $x,y$ are
adjacent in $\WG(R)$ when there exist nonzero $r\in\Ann(x)$ and
$s\in\Ann(y)$ with $rs=0$. This is a strictly weaker condition than
$xy=0$, so $\Gamma(R)$ is a spanning subgraph of $\WG(R)$. The graph
$\WG(R)$ encodes the annihilator-ideal structure of $R$ more thoroughly
than $\Gamma(R)$.

The starting point of this paper is a structural description of
$\WG(R)$ established in \cite{NNB} in the course of studying its
chromatic number. Write $R\cong\prod_{i=1}^{n} R_i$ for the
decomposition into finite local rings, unique up to isomorphism and
reordering of the factors (Proposition \ref{prop:artinian-decomp}),
and identify each $x\in R$ with its tuple $(x_1,\ldots,x_n)$,
$x_i\in R_i$. Let $F=\{i : R_i\text{ is a field}\}$ be the set of
indices of the field factors. Partition $\Z(R)^{\ast}$ into
\[
V_k = \{x\in\Z(R)^{\ast} : x_k=0 \text{ and } x_i\in\U(R_i)\text{ for all } i\neq k\}
\quad (k\in F),
\qquad V_0 = \Z(R)^{\ast}\setminus\bigcup_{k\in F} V_k.
\]
Combining \cite[Theorem 4.2 and Lemma 4.1]{NNB} (where the partition
appears as $A=\bigcup_{i\in F} A_i$ together with $B$, with our $V_k$
corresponding to their $A_i$ and $V_0$ to their $B$) gives that
$\WG(R)$ is the complete multipartite graph with parts
$\{V_k\}_{k\in F}$ together with each vertex of $V_0$ as a singleton
part. When $V_0\neq\varnothing$, this is the join of a clique on
$V_0$ with the complete $|F|$-partite graph having parts
$\{V_k\}_{k\in F}$, the formulation used in \cite{NNB} (denoted there
$K_M\vee H_k$, with $M=|V_0|$). We restate it as a single complete
multipartite graph (Theorem \ref{thm:structure}), valid uniformly
whether $V_0$ is nonempty or empty, which is the form suited to
spectral computation.

The Laplacian spectrum of $\WG(\Z_n)$ has been determined by Shariq,
Mathil, and Kumar \cite{SMK}, who also proved that $\WG(\Z_n)$ is
Laplacian integral. Their approach uses the divisor-based partition
$\bigsqcup A_{d_i}$, with $A_{d_i}=\{x\in\Z_n : \gcd(x,n)=d_i\}$, and
the generalized-join formula of Cardoso et al.\ \cite{CDMR}. The
present paper extends those results from $\Z_n$ to \emph{arbitrary
finite commutative rings} $R$ by working directly with the local-ring
decomposition and recognizing $\WG(R)$ as a complete multipartite
graph (Theorem \ref{thm:structure}); the Laplacian spectrum then
follows from the multipartite Laplacian polynomial (Theorem
\ref{thm:multipartite-charpoly}) in a single step, in a form that
specializes to \cite{SMK} for $R=\Z_n$. As parallel work for the
zero-divisor graph $\Gamma$, Chattopadhyay, Patra, and Sahoo
\cite{CPS} studied the Laplacian eigenvalues of $\Gamma(\Z_n)$,
proving in particular that $\Gamma(\Z_{p^t})$ is Laplacian integral
and expressing the Laplacian spectral radius and algebraic
connectivity (for most $n$) via a vertex-weighted Laplacian on the
proper-divisor graph of $n$; Mönius \cite{Monius} treated the
adjacency eigenvalues of $\Gamma(R)$ for general finite $R$ via a
graph product.

Combining Theorem \ref{thm:structure} with the Laplacian spectrum of a
complete multipartite graph (Theorem \ref{thm:multipartite-charpoly},
derived here from Mohar's join formula) yields the full closed-form
Laplacian spectrum of $\WG(R)$ (Theorem \ref{thm:main}); this
generalizes \cite[Theorem 4.4]{SMK} from $R=\Z_n$ to arbitrary finite
commutative $R$. Immediate consequences include Laplacian integrality
of $\WG(R)$ for every finite commutative ring $R$ (Corollary
\ref{cor:integral}) and a bound of $|F|+2$ on the number of distinct
Laplacian eigenvalues (Corollary \ref{cor:distinct}). We also derive
explicit closed-form formulas for the algebraic connectivity
(Corollary \ref{cor:alg-conn}) and the number of spanning trees
(Corollary \ref{cor:spanning-trees}) of $\WG(R)$; these follow from
the spectrum and the Matrix-Tree Theorem. As a special case, we
recover the Laplacian spectrum of $\WG(\Z_n)$ in compact closed form
(Theorem \ref{thm:Zn}).

Section \ref{sec:adj} treats the adjacency spectrum of $\WG(R)$,
which in contrast to the Laplacian case is generally not integral.

\medskip
The paper is organized as follows. Section \ref{sec:prelim} fixes
notation and gathers the ring-theoretic and spectral facts we use.
Section \ref{sec:structure} restates the structure theorem in the
unified multipartite form; for completeness we include a streamlined
proof in our notation, since the spectral arguments depend on the
explicit identification of parts. Section \ref{sec:laplacian} derives
the Laplacian spectrum and its consequences. Section \ref{sec:Zn}
specializes to $\Z_n$ and works out small cases. Section \ref{sec:adj}
treats the adjacency spectrum.

\section{Preliminaries}\label{sec:prelim}

\subsection*{Rings.}
All rings are commutative with identity. For a ring $R$, $\U(R)$
denotes its group of units and $\Z(R)$ its set of zero-divisors
(including $0$). We set $\Z(R)^{\ast} = \Z(R)\setminus\{0\}$. The
annihilator of $x\in R$ is $\Ann(x) = \{r\in R : rx=0\}$. We recall
the following standard facts, all of which can be found in
Atiyah--Macdonald \cite{AM}.

\begin{proposition}[{\cite[Theorem 8.7]{AM}}]\label{prop:artinian-decomp}
An Artin ring $R$ is uniquely \textup{(}up to isomorphism\textup{)} a
finite direct product of Artin local rings:
$R\cong\prod_{i=1}^{n} R_i$.
\end{proposition}

\begin{proposition}\label{prop:local-nonunits}
A ring $R$ is local if and only if the set of its non-units is an
ideal; in that case this set is the unique maximal ideal of $R$.
\end{proposition}

\begin{proof}
Suppose $R$ is local with maximal ideal $\m$. By
\cite[Corollary 1.5]{AM} every non-unit of $R$ lies in some maximal
ideal, hence in $\m$ (the only one); and no element of $\m$ is a
unit, since $\m\neq(1)$. Thus the set of non-units equals $\m$, an
ideal. Conversely, if the non-units form an ideal $\m$, then
$\m\neq(1)$ and every element of $R\setminus\m$ is a unit, so $R$ is
local with maximal ideal $\m$ by \cite[Proposition 1.6(i)]{AM}.
\end{proof}

\begin{lemma}\label{lem:m-is-nilradical}
Let $(R,\m)$ be an Artin local ring, and let
$\mathfrak{N}(R)=\{\,x\in R : x^{n}=0\text{ for some }n\geq 1\,\}$
denote the nilradical of $R$, that is, the set of nilpotent elements.
Then $\m$ is the unique prime ideal of $R$, and $\m=\mathfrak{N}(R)$.
\end{lemma}

\begin{proof}
In an Artin ring every prime ideal is maximal
\cite[Proposition 8.1]{AM}. Since $R$ is local it has exactly one
maximal ideal, namely $\m$; hence $\m$ is the only prime ideal of $R$.
By \cite[Proposition 1.8]{AM} the nilradical is the intersection of
all prime ideals of $R$, which here is the single ideal $\m$.
Therefore $\m=\mathfrak{N}(R)$.
\end{proof}

\begin{proposition}\label{prop:zd=nonunit}
Let $(R,\m)$ be an Artin local ring. Then an element of $R$ is a
non-unit if and only if it is a zero-divisor; consequently
$\m=\Z(R)=R\setminus\U(R)$.
\end{proposition}

\begin{proof}
Let $x$ be a non-unit. If $x=0$ then $x\in\Z(R)$ by convention.
Suppose $x\neq 0$. By Proposition \ref{prop:local-nonunits},
$x\in\m$, and by Lemma \ref{lem:m-is-nilradical},
$\m=\mathfrak{N}(R)$, so $x$ is nilpotent. Let $s\geq 2$ be minimal
with $x^{s}=0$ (here $s\geq 2$ since $x\neq 0$). Then
$x\cdot x^{s-1}=0$ with $x^{s-1}\neq 0$, so $x$ is a zero-divisor.

Conversely, a zero-divisor is never a unit, so every zero-divisor is
a non-unit.

Hence the non-units are exactly the zero-divisors, and by Proposition
\ref{prop:local-nonunits} this set is $\m$; that is,
$\m=\Z(R)=R\setminus\U(R)$.
\end{proof}

\begin{proposition}\label{prop:nilpotent}
Let $(R,\m)$ be an Artin local ring. Then $\m$ is nilpotent: there is
a smallest integer $t\geq 1$ with $\m^{t}=0$, called the
\emph{nilpotency index} of $\m$. Moreover $R$ is a field if and only
if $t=1$.
\end{proposition}

\begin{proof}
By \cite[Proposition 8.4]{AM} the nilradical of an Artin ring is
nilpotent. By Lemma \ref{lem:m-is-nilradical}, $\m=\mathfrak{N}(R)$,
so $\m$ is nilpotent; let $t\geq 1$ be smallest with $\m^{t}=0$. If
$t=1$ then $\m=0$, so every nonzero element of $R$ is a unit and $R$
is a field; conversely, if $R$ is a field then $\m=0$ and $t=1$.
\end{proof}

\begin{remark}\label{rem:finite-artinian}
A finite ring has only finitely many ideals, so it trivially
satisfies the descending chain condition and is Artin; hence the
results \ref{prop:artinian-decomp}--\ref{prop:nilpotent} above apply
to every finite ring.
\end{remark}

\subsection*{The Laplacian matrix.}
Let $G=(V,E)$ be a simple graph on $|V|=N$ vertices. The
\emph{Laplacian matrix} of $G$ is $L(G) = D(G) - A(G)$, where $A(G)$ is
the adjacency matrix and $D(G)$ is the diagonal matrix of vertex
degrees. The matrix $L(G)$ is symmetric and positive semidefinite, so
its eigenvalues are real and nonnegative; we denote them
$0=\mu_1(G)\leq\mu_2(G)\leq\cdots\leq\mu_N(G)$. The multiplicity of $0$
equals the number of connected components of $G$. We call $G$
\emph{Laplacian integral} if every $\mu_i(G)$ is an integer.

We will use Mohar's formula for the Laplacian characteristic
polynomial of a graph join. Recall that the \emph{join} $G_1\vee G_2$
of two graphs is the graph obtained from their disjoint union by
adding all edges $uv$ with $u\in V(G_1)$ and $v\in V(G_2)$. (Mohar
writes this operation as $G_1\ast G_2$; it is exactly the join $\vee$
used here.)

\begin{theorem}[Mohar {\cite[Corollary 3.7]{Mohar}}]\label{thm:mohar}
Let $G_1,G_2$ be graphs on $n_1,n_2$ vertices respectively. Then
\[
\mu_L(G_1\vee G_2,\,x) \;=\;
\frac{x(x-n_1-n_2)}{(x-n_1)(x-n_2)}\,
\mu_L(G_1,\,x-n_2)\,\mu_L(G_2,\,x-n_1).
\]
\end{theorem}

\section{Structure of \texorpdfstring{$\WG(R)$}{WGamma(R)}}\label{sec:structure}

Throughout this section $R$ denotes a finite commutative ring with
identity and $R\cong\prod_{i=1}^{n} R_i$ its decomposition into a
product of finite local rings (Proposition \ref{prop:artinian-decomp}).
We identify each $x\in R$ with its tuple $(x_1,\ldots,x_n)$, where
$x_i\in R_i$. We write $\m_i$ for the maximal ideal of $R_i$ and
$t_i$ for its nilpotency index, which is well defined by Proposition
\ref{prop:nilpotent}; recall from that proposition that $R_i$ is a
field if and only if $t_i=1$. Let $F$ be the index set
\[
F = \{\, i\in\{1,\ldots,n\} : R_i \text{ is a field}\,\},
\]
and define, for each $k\in F$,
\[
V_k = \{\, x\in\Z(R)^{\ast} : x_k=0 \text{ and } x_i\in\U(R_i)\text{ for all } i\neq k\,\},
\]
and
\[
V_0 = \Z(R)^{\ast}\setminus\bigcup_{k\in F} V_k.
\]
We call coordinate $i$ a \emph{zero-divisor coordinate} of
$x=(x_1,\ldots,x_n)$ if $x_i\in\Z(R_i)$.

\subsection*{An annihilator lemma.}

\begin{lemma}\label{lem:annihilator}
Let $(R,\m)$ be an Artin local ring that is not a field. By
Proposition \ref{prop:nilpotent} its nilpotency index satisfies
$t\geq 2$, and $\m^{\,t-1}\neq 0$ by minimality of $t$. Any nonzero
$u\in\m^{\,t-1}$ satisfies:
\begin{enumerate}[label=\textup{(\arabic*)}]
\item $u^2=0$;
\item $ux=0$ for every zero-divisor $x\in R$.
\end{enumerate}
\end{lemma}

\begin{proof}
By Proposition \ref{prop:zd=nonunit} every zero-divisor lies in $\m$.
For $x\in\m$, $ux\in\m^{\,t-1}\cdot\m=\m^{\,t}=0$, proving (2). For (1),
$u\in\m^{\,t-1}$ gives $u^2\in\m^{\,2(t-1)}$. Since $t\geq 2$, we have
$2(t-1)\geq t$, so $\m^{\,2(t-1)}\subseteq\m^{\,t}=0$, and $u^2=0$.
\end{proof}

\subsection*{The multipartite structure.}

The following structural description of $\WG(R)$ is
\cite[Theorem 4.2 and Lemma 4.1]{NNB}. In \cite{NNB} it is established
in the course of computing $\omega(\WG(R))$ and $\chi(\WG(R))$, and is
phrased there as the join of a clique on $V_0$ with the complete
$|F|$-partite graph having parts $\{V_k\}_{k\in F}$ (denoted
$K_M\vee H_k$ with $M=|V_0|$). We restate it as a single complete
multipartite graph, identifying the parts explicitly with the sets
$V_k$ and $V_0$ and treating the cases $V_0\neq\varnothing$ and
$V_0=\varnothing$ uniformly; we include a streamlined proof in our
notation, since the spectral computations in subsequent sections
depend on this explicit identification.

\begin{theorem}\label{thm:structure}
Let $R$ be a finite commutative ring with $\WG(R)$ nonempty
\textup{(}the case $\Z(R)^{\ast}=\varnothing$, i.e.\ $R$ a field, being
trivial\textup{)}. Then $\WG(R)$ is the complete
multipartite graph whose parts are $\{V_k : k\in F\}$ together with each
vertex of $V_0$ as a singleton part. Equivalently:
\begin{enumerate}[label=\textup{(\arabic*)}]
\item $\WG(R)[V_k]$ is an independent set for each $k\in F$;
\item $\WG(R)[V_0]$ is a clique;
\item for each $k\in F$, every vertex of $V_k$ is adjacent to every
vertex of $\Z(R)^{\ast}\setminus V_k$.
\end{enumerate}
\end{theorem}

\begin{proof}
For $x=(x_1,\ldots,x_n)$ and $i\in\{1,\ldots,n\}$, write
$\Ann_i(x_i)=\{r_i\in R_i : r_i x_i=0\}$. Under the product
decomposition $R\cong\prod R_i$, the annihilator of $x$ factors
coordinate-wise: $\Ann(x)=\prod_{i=1}^{n}\Ann_i(x_i)$, i.e.\ $r\in\Ann(x)$
iff $r_i x_i = 0$ for every $i$.

\medskip\noindent
\textbf{(1) Independence of $V_k$.} Fix $k\in F$ and distinct
$x,y\in V_k$. Then $x_k=y_k=0$ and $x_i,y_i\in\U(R_i)$ for $i\neq k$, so
$\Ann_i(x_i)=\Ann_i(y_i)=\{0\}$ for $i\neq k$ and
$\Ann_k(x_k)=\Ann_k(y_k)=R_k$. Hence
\[
\Ann(x) = \Ann(y) = \{\,r\in R : r_i=0\text{ for }i\neq k,\;r_k\in R_k\,\}.
\]
For any nonzero $r\in\Ann(x)$ and $s\in\Ann(y)$ we have
$r_k,s_k\in R_k\setminus\{0\}=\U(R_k)$ since $R_k$ is a field. Thus
$(rs)_k=r_k s_k\neq 0$, and $rs\neq 0$. Hence $x,y$ are non-adjacent.

\medskip\noindent
\textbf{(2) $V_0$ is a clique.} Let $x,y$ be distinct vertices of
$V_0$ and set
\[
A=\{\,i : x_i\in\Z(R_i)\,\},\qquad B=\{\,i : y_i\in\Z(R_i)\,\}.
\]
Both $A$ and $B$ are nonempty since $x,y$ are zero-divisors. We
construct nonzero $r\in\Ann(x)$, $s\in\Ann(y)$ with $rs=0$.

\emph{Case 1: $A\cap B=\varnothing$.} Choose $i\in A$, $j\in B$, and
nonzero $\alpha_i\in\Ann_i(x_i)$, $\beta_j\in\Ann_j(y_j)$ (such
$\alpha_i$ exists since $x_i\in\Z(R_i)$: if $x_i=0$ take $\alpha_i=1$,
otherwise a nonzero annihilator exists by definition of zero-divisor;
similarly for $\beta_j$). Define $r$ by
$r_i=\alpha_i$ and zeros elsewhere; define $s$ by $s_j=\beta_j$ and
zeros elsewhere. Then $r\in\Ann(x)$, $s\in\Ann(y)$ are nonzero, and
since $i\neq j$, $(rs)_\ell=0$ for all $\ell$, so $rs=0$.

\emph{Case 2: $A\cap B\neq\varnothing$.} If both $A\setminus B$ and
$B\setminus A$ are nonempty, pick $i\in A\setminus B$ and
$j\in B\setminus A$; then $i\neq j$ and the construction of Case 1
applies verbatim. Otherwise, without loss of generality
$A\subseteq B$.

\quad\emph{Case 2a: $|A|=1$,} say $A=\{i\}$. We claim $i\notin F$.
Indeed, if $i\in F$, then $R_i$ is a field, so $x_i\in\Z(R_i)$ forces
$x_i=0$. For each $j\neq i$ we have $j\notin A$, hence
$x_j\notin\Z(R_j)$; since $R_j$ is Artin local, Proposition
\ref{prop:zd=nonunit} applied to $R_j$ gives $x_j\in\U(R_j)$. So
$x_i=0$ and $x_j\in\U(R_j)$ for all $j\neq i$, which together with
$i\in F$ gives $x\in V_i$, contradicting $x\in V_0$. So $R_i$ is local
but not a field; by Proposition \ref{prop:nilpotent}, $t_i\geq 2$, and
$\m_i^{\,t_i-1}\neq 0$ by minimality of $t_i$.
Pick a nonzero $u_i\in\m_i^{\,t_i-1}$. By Lemma \ref{lem:annihilator}(2),
$u_i$ annihilates every zero-divisor of $R_i$, in particular $u_i x_i=0$
and $u_i y_i=0$ (the latter because $i\in A\subseteq B$, so
$y_i\in\Z(R_i)$). Define $r$ by $r_i=u_i$, zeros elsewhere, and $s$
similarly. Then $r\in\Ann(x)$, $s\in\Ann(y)$, both nonzero. Both
are supported only at coordinate $i$, so $rs$ is also supported only
at $i$, with $(rs)_i=u_i^{2}=0$ by Lemma \ref{lem:annihilator}(1).
Hence $rs=0$.

\quad\emph{Case 2b: $|A|\geq 2$.} Pick distinct $i,j\in A\subseteq B$,
and nonzero $\alpha_i\in\Ann_i(x_i)$, $\beta_j\in\Ann_j(y_j)$. Define
$r$ by $r_i=\alpha_i$ and zeros elsewhere; $s$ by $s_j=\beta_j$ and
zeros elsewhere. Both are nonzero and $rs=0$ because $i\neq j$.

\medskip\noindent
\textbf{(3) Cross edges.}

\emph{Edges between $V_k$ ($k\in F$) and $V_0$.} Let $x\in V_k$ and
$y\in V_0$. We first claim that there exists $j\neq k$ with
$y_j\notin\U(R_j)$. Suppose not; then $y_j\in\U(R_j)$ for every
$j\neq k$, so every coordinate of $y$ except possibly the $k$-th is a
unit. Since $y\in V_0\subseteq\Z(R)^{\ast}$, the element $y$ is a
zero-divisor of $R$, hence not a unit; as an element of a product
ring is a unit iff all its coordinates are units, $y_k$ cannot be a
unit either, so $y_k\notin\U(R_k)$. But $R_k$ is a field
(as $k\in F$), and the only non-unit of a field is $0$; therefore
$y_k=0$. Thus $y_k=0$ and $y_i\in\U(R_i)$ for all $i\neq k$, which is
exactly the defining condition for $y\in V_k$. This contradicts
$y\in V_0$, since $V_0$ is disjoint from $V_k$. This proves the claim:
some $j\neq k$ has $y_j\notin\U(R_j)$.

Fix such a $j$. Since $R_j$ is Artin local, Proposition
\ref{prop:zd=nonunit} gives $y_j\in\Z(R_j)$, so $\Ann_j(y_j)\neq 0$;
pick a nonzero $\beta_j\in\Ann_j(y_j)$. Define
$r$ by $r_k=1$, zeros elsewhere (so $r\in\Ann(x)$ since $x_k=0$); define
$s$ by $s_j=\beta_j$, zeros elsewhere. Both are nonzero, and $rs=0$
because $j\neq k$.

\emph{Edges between $V_k$ and $V_\ell$ ($k\neq\ell$, both in $F$).} Let
$x\in V_k$, $y\in V_\ell$. Define $r$ by $r_k=1$, zeros elsewhere, and
$s$ by $s_\ell=1$, zeros elsewhere. Both are nonzero and $rs=0$ because
$k\neq\ell$.

\medskip
Combining (1)--(3), $\WG(R)$ is the complete multipartite graph with
parts $\{V_k\}_{k\in F}$ together with each element of $V_0$ as a
singleton part.
\end{proof}

\begin{remark}\label{rem:sizes}
The part sizes are
\[
|V_k| \;=\; \prod_{\substack{i=1\\ i\neq k}}^{n}|\U(R_i)|, \qquad k\in F,
\quad\text{and}\quad
|V_0| \;=\; |\Z(R)^{\ast}| - \sum_{k\in F}|V_k|.
\]
\end{remark}

\begin{example}\label{ex:Z6}
Let $R=\Z_2\times\Z_3$, so both factors are fields and $F=\{1,2\}$.
Then $\U(\Z_2)=\{1\}$, $\U(\Z_3)=\{1,2\}$, and
\begin{align*}
V_1 &= \{(0,1),(0,2)\}, & |V_1| &= |\U(\Z_3)|=2,\\
V_2 &= \{(1,0)\}, & |V_2| &= |\U(\Z_2)|=1,\\
V_0 &= \varnothing.
\end{align*}
Thus $\WG(\Z_2\times\Z_3) = K(2,1)$, the path on three
vertices. (This agrees with $\WG(\Z_6)$ since $\Z_6\cong\Z_2\times\Z_3$.)
\end{example}

\begin{example}\label{ex:Z12}
Let $R=\Z_{12}\cong\Z_4\times\Z_3$, with $R_1=\Z_4$ and $R_2=\Z_3$.
The factor $R_1=\Z_4$ is local but \emph{not} a field (it has the
zero-divisor $2$), so $1\notin F$; the factor $R_2=\Z_3$ is a field,
so $2\in F$. Hence $F=\{2\}$. We have $\U(\Z_4)=\{1,3\}$,
$\U(\Z_3)=\{1,2\}$, and $|\Z(\Z_{12})^{\ast}|=12-\varphi(12)-1=7$. Then
\[
V_2 = \{(1,0),(3,0)\},\qquad |V_2|=2;\qquad |V_0|=7-2=5.
\]
Hence $\WG(\Z_{12}) = K(2,\underbrace{1,1,1,1,1}_{5})$, a complete
multipartite graph on $7$ vertices with one part of size $2$ and five
singleton parts.
\end{example}

\section{Laplacian spectrum of \texorpdfstring{$\WG(R)$}{WGamma(R)}}\label{sec:laplacian}

The Laplacian characteristic polynomial of a complete multipartite
graph is a known result in spectral graph theory; it follows, for
instance, from the generalized-join Laplacian formula of Cardoso et
al.\ \cite{CDMR} applied to edgeless graphs. We include a short
self-contained derivation by induction from Mohar's join formula
(Theorem \ref{thm:mohar} above; \cite[Corollary 3.7]{Mohar}) because
the explicit polynomial form is needed for the spectral computations
in the rest of this section.

\begin{theorem}\label{thm:multipartite-charpoly}
The Laplacian characteristic polynomial of the complete multipartite
graph $K(p_1,\ldots,p_n)$ on $N=\sum_{i=1}^{n}p_i$ vertices is
\[
\mu_L(x) \;=\; x\,(x-N)^{\,n-1}\prod_{i=1}^{n}\bigl(x-(N-p_i)\bigr)^{\,p_i-1}.
\]
\end{theorem}

\begin{proof}
Write $K(p_1,\ldots,p_n)=\overline{K}_{p_1}\vee\overline{K}_{p_2}\vee\cdots\vee\overline{K}_{p_n}$,
where $\overline{K}_m$ denotes the edgeless graph on $m$ vertices.
The Laplacian of $\overline{K}_m$ is the zero matrix, so
$\mu_L(\overline{K}_m,x)=x^{m}$. Proceed by induction on $n$.

\emph{Base case $n=1$:} $K(p_1)=\overline{K}_{p_1}$ gives
$\mu_L(x)=x^{p_1}$, matching $x\cdot(x-p_1)^{0}\cdot(x-0)^{p_1-1}=x^{p_1}$.

\emph{Inductive step:} Let
$G_{n-1}=K(p_1,\ldots,p_{n-1})$ on $N_{n-1}=N-p_n$ vertices and assume
\[
\mu_L(G_{n-1},x)=x\,(x-N_{n-1})^{\,n-2}\prod_{j=1}^{n-1}\bigl(x-(N_{n-1}-p_j)\bigr)^{\,p_j-1}.
\]
Apply Theorem \ref{thm:mohar} to $G_n=G_{n-1}\vee\overline{K}_{p_n}$
with $n_1=N_{n-1}$, $n_2=p_n$:
\[
\mu_L(G_n,x)=\frac{x(x-N)}{(x-N_{n-1})(x-p_n)}\,\mu_L(G_{n-1},x-p_n)\,\mu_L(\overline{K}_{p_n},x-N_{n-1}).
\]
Substituting $\mu_L(\overline{K}_{p_n},x-N_{n-1})=(x-N_{n-1})^{p_n}=(x-(N-p_n))^{p_n}$ and
\[
\mu_L(G_{n-1},x-p_n)=(x-p_n)(x-N)^{\,n-2}\prod_{j=1}^{n-1}\bigl(x-(N-p_j)\bigr)^{\,p_j-1},
\]
and simplifying yields
\[
\mu_L(G_n,x)=x\,(x-N)^{\,n-1}\prod_{i=1}^{n}\bigl(x-(N-p_i)\bigr)^{\,p_i-1}. \qedhere
\]
\end{proof}

\begin{corollary}\label{cor:multipartite-spec}
The Laplacian eigenvalues of $K(p_1,\ldots,p_n)$ are:
\begin{enumerate}[label=\textup{(\arabic*)}]
\item $0$ with multiplicity $1$;
\item $N=\sum_{i=1}^{n}p_i$ with multiplicity $n-1$ (so $N$ is an
eigenvalue precisely when $n\geq 2$);
\item for each $i$, the value $N-p_i$ with multiplicity $p_i-1$.
\end{enumerate}
If $p_i=1$ for some $i$, the corresponding factor
$(x-(N-1))^{0}=1$ contributes no eigenvalue.
\end{corollary}

\subsection*{Main theorem.}

\begin{theorem}\label{thm:main}
Let $R\cong\prod_{i=1}^{n} R_i$ be a finite commutative ring and set
$N=|\Z(R)^{\ast}|$. The Laplacian eigenvalues of $\WG(R)$ are:
\begin{enumerate}[label=\textup{(\arabic*)}]
\item $0$ with multiplicity $1$;
\item $N$ with multiplicity $|F|+|V_0|-1$;
\item for each $k\in F$, the value $N-|V_k|$ with multiplicity $|V_k|-1$.
\end{enumerate}
\end{theorem}

\begin{proof}
By Theorem \ref{thm:structure}, $\WG(R)=K(p_1,\ldots,p_M)$ where
$M=|F|+|V_0|$, with $p_k=|V_k|$ for $k\in F$ and $p=1$ for each
singleton part corresponding to a vertex of $V_0$. The total
number of vertices is
$N=\sum_{k\in F}|V_k|+|V_0|=|\Z(R)^{\ast}|$.

Applying Corollary \ref{cor:multipartite-spec}:
$0$ appears with multiplicity $1$; $N$ appears with multiplicity
$M-1=|F|+|V_0|-1$; for each $k\in F$, $N-|V_k|$ appears with
multiplicity $|V_k|-1$; each singleton part contributes multiplicity
$p-1=0$, so no further eigenvalues arise.
\end{proof}

\begin{corollary}\label{cor:integral}
$\WG(R)$ is Laplacian integral for every finite commutative ring $R$.
\end{corollary}

\begin{proof}
All eigenvalues listed in Theorem \ref{thm:main} are integers.
\end{proof}

\begin{corollary}\label{cor:distinct}
The number of distinct Laplacian eigenvalues of $\WG(R)$ is at most
$|F|+2$.
\end{corollary}

\begin{proof}
By Theorem \ref{thm:main} the distinct eigenvalues lie among $0$, $N$,
and $\{N-|V_k| : k\in F\}$.
\end{proof}

\begin{remark}
The bound $|F|+2$ is attained whenever the values $|V_k|$ for
$k\in F$ are pairwise distinct, all at least $2$, and
$|F|+|V_0|\geq 2$ (then the eigenvalues $0$, $N$, and the $|F|$
distinct values $N-|V_k|$ are all present and distinct). It can fail
to be attained, for instance, if two distinct $k,\ell\in F$ have
$|V_k|=|V_\ell|$ (the eigenvalues $N-|V_k|$ and $N-|V_\ell|$ then
coincide), or if some $|V_k|=1$ (the value $N-|V_k|=N-1$ then has
multiplicity $|V_k|-1=0$ and does not appear in the spectrum). Note
that $|V_k|\geq 1$ always, since $|V_k|=\prod_{i\neq k}|\U(R_i)|$ and
$1\in\U(R_i)$ for every nontrivial ring $R_i$.
\end{remark}

\begin{remark}\label{rem:all-nonfields}
If every $R_i$ is a non-field (i.e.\ $F=\varnothing$), then
$V_0=\Z(R)^{\ast}$ and Theorem \ref{thm:structure} reduces to
$\WG(R)=K_N$. Theorem \ref{thm:main} then gives Laplacian spectrum
$\{0, N^{(N-1)}\}$, recovering the classical formula for $\mu_L(K_N)$.
\end{remark}

\begin{remark}\label{rem:degenerate}
We adopt the convention that $\WG(R)$ is the graph with vertex set
$\Z(R)^{\ast}$; when $\Z(R)^{\ast}=\varnothing$, the Laplacian
spectrum is the empty multiset. In particular, if $R$ is a field then
$\Z(R)^{\ast}=\varnothing$ and there is nothing to compute. At the
opposite extreme, if $|\Z(R)^{\ast}|=1$ (e.g.\ $R=\Z_4$ or
$R=\F_2[x]/(x^2)$), then $N=1$, $|F|=0$, $|V_0|=1$, and Theorem
\ref{thm:main} gives the spectrum $\{0\}$: the multiplicity of $N=1$
is $|F|+|V_0|-1=0$, and no eigenvalues of the form $N-|V_k|$ arise
($F=\varnothing$).
\end{remark}

\subsection*{Consequences: algebraic connectivity and spanning trees.}

The \emph{algebraic connectivity} $\mu_2(G)$ of $G$ is the second
smallest Laplacian eigenvalue, a much-studied parameter introduced by
Fiedler \cite{Fiedler}. The \emph{number of spanning trees} $\tau(G)$
is computable from the nonzero Laplacian eigenvalues via the
Matrix-Tree Theorem:
\[
\tau(G) \;=\; \frac{1}{N}\prod_{i=2}^{N}\mu_i(G).
\]
Both invariants are immediate from Theorem \ref{thm:main}.

\begin{corollary}\label{cor:alg-conn}
Let $R\cong\prod_{i=1}^{n} R_i$ be a finite commutative ring with
$N=|\Z(R)^{\ast}|\geq 2$, and set
$p_{\max} := \max\bigl(\{|V_k| : k\in F\}\cup\{1\}\bigr)$. Then the
algebraic connectivity of $\WG(R)$ is
\[
\mu_2(\WG(R)) \;=\;
\begin{cases}
N - p_{\max}, & \text{if } p_{\max}\geq 2,\\
N, & \text{if } p_{\max}=1.
\end{cases}
\]
In particular, if $F=\varnothing$ then $\WG(R)=K_N$ and
$\mu_2(\WG(R))=N$.
\end{corollary}

\begin{proof}
By Theorem \ref{thm:main}, the positive Laplacian eigenvalues of
$\WG(R)$ are: the value $N$ (with multiplicity $|F|+|V_0|-1$), and the
values $N-|V_k|$ for $k\in F$ with $|V_k|\geq 2$ (with multiplicity
$|V_k|-1$); values $N-|V_k|$ corresponding to $|V_k|=1$ do not appear,
since their multiplicity is $|V_k|-1=0$.

If some $|V_k|\geq 2$, then $p_{\max}\geq 2$ and the smallest positive
eigenvalue is $\min_{k\in F,\,|V_k|\geq 2}(N-|V_k|)=N-p_{\max}$, which
is at most $N$ (the inequality is strict since $p_{\max}\geq 2$). Hence
$\mu_2(\WG(R))=N-p_{\max}$.

If $|V_k|=1$ for every $k\in F$ (which includes the case $F=\varnothing$),
then $p_{\max}=1$ and no eigenvalue of the form $N-|V_k|$ appears; the
only positive Laplacian eigenvalue is $N$, so $\mu_2(\WG(R))=N$.

The case $F=\varnothing$ gives $V_0=\Z(R)^{\ast}$ and $\WG(R)=K_N$ by
Theorem \ref{thm:structure}; the algebraic connectivity of $K_N$ is
indeed $N$.
\end{proof}

\begin{remark}
The hypothesis $N\geq 2$ in Corollary \ref{cor:alg-conn} is essential:
for $N\leq 1$, the graph has fewer than two vertices and $\mu_2$ is
undefined.
\end{remark}

\begin{corollary}\label{cor:spanning-trees}
Let $R$ be a finite commutative ring with $\WG(R)$ nonempty
\textup{(}equivalently, $\Z(R)^{\ast}\neq\varnothing$, i.e.\ $R$ is not
a field\textup{)}, and write $N=|\Z(R)^{\ast}|$. If $N=1$ then
$\WG(R)$ is a single vertex and $\tau(\WG(R))=1$. If $N\geq 2$ then
\[
\tau(\WG(R)) \;=\; N^{\,|F|+|V_0|-2}\prod_{k\in F}\bigl(N-|V_k|\bigr)^{\,|V_k|-1},
\]
where the empty product over $F=\varnothing$ is taken to be $1$.
\end{corollary}

\begin{proof}
If $N=1$ the graph $\WG(R)$ has one vertex and trivially one spanning
tree. Suppose $N\geq 2$. By Theorem \ref{thm:structure}, $\WG(R)$ is
a complete multipartite graph with at least two parts \textup{(}since
$N\geq 2$\textup{)}, hence has diameter at most $2$ and is connected;
the Matrix-Tree Theorem then gives
\[
\tau(\WG(R)) \;=\; \tfrac{1}{N}\prod_{i=2}^{N}\mu_i(\WG(R)).
\]
Substituting the eigenvalue multiplicities from Theorem
\ref{thm:main},
\[
\tau(\WG(R)) \;=\; \tfrac{1}{N}\cdot N^{|F|+|V_0|-1}\cdot \prod_{k\in F}(N-|V_k|)^{|V_k|-1}
            \;=\; N^{|F|+|V_0|-2}\prod_{k\in F}(N-|V_k|)^{|V_k|-1}. \qedhere
\]
\end{proof}

\section{Specialization to \texorpdfstring{$\Z_n$}{Z\_n}}\label{sec:Zn}

Write $n=p_1^{a_1}\cdots p_m^{a_m}$ with distinct primes $p_1,\ldots,p_m$
and $a_i\geq 1$, so that $\Z_n\cong\prod_{i=1}^{m}\Z_{p_i^{a_i}}$ by the
Chinese Remainder Theorem. Throughout this section we identify the
local-ring decomposition of $\Z_n$ with this product, so that index $k$
corresponds to the prime $p_k$ and the local factor $\Z_{p_k^{a_k}}$.
The factor $\Z_{p_i^{a_i}}$ is a field iff $a_i=1$, so
\[
F = \{\, i : a_i=1\,\},\qquad |\U(\Z_{p_i^{a_i}})| = \varphi(p_i^{a_i}).
\]
By Remark \ref{rem:sizes}, for $k\in F$,
\[
|V_k| \;=\; \prod_{\substack{i=1\\ i\neq k}}^{m}\varphi(p_i^{a_i})
       \;=\; \varphi(n/p_k),
\]
using multiplicativity of $\varphi$ on coprime arguments (here $n/p_k$
is well-defined because $a_k=1$).

We can now read off the Laplacian spectrum of $\WG(\Z_n)$ as an
immediate corollary of the general Theorem \ref{thm:main}, requiring
only the substitution of these values of $|V_k|$ and $N$. Our
derivation differs methodologically from the proof of
\cite[Theorem 4.4]{SMK}: \cite{SMK} works through the divisor-class
partition and the generalized-join Laplacian formula of Cardoso et
al.\ \cite{CDMR}, while in our approach the result follows from the
complete multipartite structure (Theorem \ref{thm:structure}) and the
multipartite Laplacian spectrum (Theorem
\ref{thm:multipartite-charpoly}) in a single step.

\begin{theorem}[cf.\ {\cite[Theorem 4.4]{SMK}}]\label{thm:Zn}
Let $n\geq 4$ be a composite integer \textup{(}so $\Z_n$ is not a
field\textup{)}, with prime factorization
$n=p_1^{a_1}\cdots p_m^{a_m}$, and let $F=\{i:a_i=1\}$. Then the
Laplacian spectrum of $\WG(\Z_n)$ consists of:
\begin{enumerate}[label=\textup{(\arabic*)}]
\item $0$ with multiplicity $1$;
\item $N=n-\varphi(n)-1$ with multiplicity
\[
|F|+|V_0|-1 \;=\; n+|F|-2-\varphi(n)-\sum_{k\in F}\varphi\!\left(\tfrac{n}{p_k}\right);
\]
\item for each $k\in F$, the value $N-\varphi(n/p_k)$ with multiplicity
$\varphi(n/p_k)-1$.
\end{enumerate}
\end{theorem}

\begin{proof}
Apply Theorem \ref{thm:main}. The total number of zero-divisors of
$\Z_n$ equals $n-\varphi(n)$, so $N=|\Z(\Z_n)^{\ast}|=n-\varphi(n)-1$.
For $k\in F$ we have $|V_k|=\varphi(n/p_k)$ as computed above. Finally,
\[
|V_0|=N-\sum_{k\in F}|V_k|=n-\varphi(n)-1-\sum_{k\in F}\varphi(n/p_k),
\]
and substituting into $|F|+|V_0|-1$ gives the multiplicity of $N$.
\end{proof}

\begin{remark}
Theorem \ref{thm:Zn} is the $\Z_n$-specialization of Theorem
\ref{thm:main} and recovers the result of Shariq, Mathil, and
Kumar \cite[Theorem 4.4]{SMK}. Their proof uses the divisor-based
partition $\Z(\Z_n)^{\ast}=\bigsqcup_{d_i} A_{d_i}$ with
$A_{d_i}=\{x : \gcd(x,n)=d_i\}$ and the generalized-join Laplacian
formula of Cardoso et al.\ \cite{CDMR}; ours obtains the same
spectrum more directly as a corollary of Theorem \ref{thm:structure}
(under the correspondence $A_{p_i}\leftrightarrow V_{k_i}$ for primes
$p_i$ with $a_i=1$, and $\bigsqcup A_{d_j}\leftrightarrow V_0$ for the
remaining divisors).
\end{remark}

\subsection*{Small cases.}

\begin{example}[$n=pq$, distinct primes]\label{ex:pq}
Here $\Z_{pq}\cong\Z_p\times\Z_q$ with $R_1=\Z_p$ and $R_2=\Z_q$, both
fields, so $F=\{1,2\}$. We get
$\varphi(n)=(p-1)(q-1)$, $N=pq-(p-1)(q-1)-1=p+q-2$,
$|V_1|=\varphi(q)=q-1$, $|V_2|=\varphi(p)=p-1$, and
$|V_0|=N-(p-1)-(q-1)=0$. So $\WG(\Z_{pq})=K_{p-1,q-1}$, the complete
bipartite graph. The Laplacian spectrum is
\[
\bigl\{\,0,\; (p+q-2)^{(1)},\; (q-1)^{(p-2)},\; (p-1)^{(q-2)}\,\bigr\}.
\]
For instance, for $n=15$ ($p=3$, $q=5$): $N=6$, $|V_1|=\varphi(5)=4$,
$|V_2|=\varphi(3)=2$, and the Laplacian spectrum is
$\{0,\, 2^{(3)},\, 4^{(1)},\, 6^{(1)}\} = \{0,2,2,2,4,6\}$;
indeed $\WG(\Z_{15})=K_{2,4}$.
\end{example}

\begin{example}[$n=pqr$, distinct primes]\label{ex:pqr}
With $F=\{1,2,3\}$, write $V_p,V_q,V_r$ for the parts indexed by the
primes $p,q,r$ respectively. The part sizes are
$|V_p|=\varphi(qr)=(q-1)(r-1)$,
$|V_q|=\varphi(pr)=(p-1)(r-1)$, $|V_r|=\varphi(pq)=(p-1)(q-1)$, and
\[
N \;=\; pqr-(p-1)(q-1)(r-1)-1 \;=\; pq+pr+qr-p-q-r.
\]
Substituting $u=p-1$, $v=q-1$, $w=r-1$ gives $N=uv+uw+vw+u+v+w$ and
$\sum_{k\in F}|V_k|=uv+uw+vw$, so
\[
|V_0| \;=\; N - \sum_{k\in F}|V_k| \;=\; u+v+w \;=\; p+q+r-3.
\]
The Laplacian spectrum of $\WG(\Z_{pqr})$ is therefore
\[
\bigl\{\,0,\; N^{(p+q+r-1)},\; (N-\varphi(qr))^{(\varphi(qr)-1)},\;
(N-\varphi(pr))^{(\varphi(pr)-1)},\; (N-\varphi(pq))^{(\varphi(pq)-1)}\,\bigr\}
\]
(the multiplicity $p+q+r-1$ of $N$ comes from $|F|+|V_0|-1 = 3+(p+q+r-3)-1$).

\medskip
For the specific case $n=30=2\cdot 3\cdot 5$ we have:
$N=21$, $|V_1|=\varphi(15)=8$, $|V_2|=\varphi(10)=4$,
$|V_3|=\varphi(6)=2$, $|V_0|=p+q+r-3=7$. The Laplacian spectrum is
\[
\bigl\{\,0,\; 21^{(9)},\; 13^{(7)},\; 17^{(3)},\; 19^{(1)}\,\bigr\}.
\]
For $n=105=3\cdot 5\cdot 7$: $N=15+21+35-15=56$, $|V_0|=12$,
$|V_1|=\varphi(35)=24$, $|V_2|=\varphi(21)=12$, $|V_3|=\varphi(15)=8$.
The Laplacian spectrum is
\[
\bigl\{\,0,\; 56^{(14)},\; 32^{(23)},\; 44^{(11)},\; 48^{(7)}\,\bigr\}.
\]
\end{example}

\begin{example}[$n=12=2^{2}\cdot 3$]\label{ex:12}
Here $F=\{2\}$ (only $\Z_3$ is a field). From Example \ref{ex:Z12},
$N=7$, $|V_2|=2$, $|V_0|=5$. The Laplacian spectrum is
\[
\bigl\{\,0,\; 7^{(5)},\; 5^{(1)}\,\bigr\}.
\]
The algebraic connectivity is $N-p_{\max}=7-2=5$, and the number of
spanning trees is
$\tau(\WG(\Z_{12}))=7^{|F|+|V_0|-2}\cdot 5^{|V_2|-1}=7^{1+5-2}\cdot 5^{2-1}=7^{4}\cdot 5 = 12005$.
\end{example}

\subsection*{Tabulated examples.}
\begin{center}
\begin{tabular}{@{}lcccl@{}}
\toprule
$n$ & $N$ & $|F|$ & $|V_0|$ & Laplacian spectrum (multiplicities in superscripts) \\
\midrule
$6$  & $3$ & $2$ & $0$ & $\{0,\,1^{(1)},\,3^{(1)}\}$\\
$10$ & $5$ & $2$ & $0$ & $\{0,\,1^{(3)},\,5^{(1)}\}$\\
$12$ & $7$ & $1$ & $5$ & $\{0,\,5^{(1)},\,7^{(5)}\}$\\
$15$ & $6$ & $2$ & $0$ & $\{0,\,2^{(3)},\,4^{(1)},\,6^{(1)}\}$\\
$30$ & $21$ & $3$ & $7$ & $\{0,\,13^{(7)},\,17^{(3)},\,19^{(1)},\,21^{(9)}\}$\\
\bottomrule
\end{tabular}
\end{center}

\section{Adjacency spectrum}\label{sec:adj}

The complete multipartite structure also determines the adjacency
spectrum of $\WG(R)$. Rather than merely substituting into the
Esser--Harary formula \cite{EH}, we give a decomposition in which each
part of the spectrum is interpreted in terms of the part-size
multiset of $\WG(R)$: the independent sets $V_k$ contribute the
eigenvalue $0$; for each distinct part-size $q$ appearing at $r_q\geq 2$
parts, the eigenvalue $-q$ arises with multiplicity $r_q-1$
\textup{(}including the special case $q=1$, where the size-$1$ parts
form a clique and contribute the eigenvalue $-1$\textup{)}; and the
remaining eigenvalues are governed by a small explicit secular
equation in the distinct part-sizes.

\begin{theorem}[Esser--Harary \cite{EH}]\label{thm:EH}
For the complete multipartite graph $G=K(q_1,\ldots,q_t)$ on
$Q=\sum_{i} q_i$ vertices, the adjacency characteristic polynomial is
\[
\mu_A(G,x) \;=\; x^{\,Q-t}\left[\prod_{i=1}^{t}(x+q_i)
\;-\;\sum_{i=1}^{t}q_i\prod_{\substack{j=1\\ j\neq i}}^{t}(x+q_j)\right].
\]
\end{theorem}

\begin{theorem}\label{thm:adj-WG}
Let $R$ be a finite commutative ring with $\WG(R)$ nonempty, and let
$F$, $V_k$ $(k\in F)$, $V_0$ and $N=|\Z(R)^{\ast}|$ be as in Section
\ref{sec:structure}. Let
$\mathcal{P}=\{|V_k| : k\in F\}\cup\{\underbrace{1,\ldots,1}_{|V_0|}\}$
be the multiset of part-sizes of $\WG(R)$ \textup{(}each $|V_k|$ for
$k\in F$ counted once, and $1$ counted with multiplicity $|V_0|$\textup{)}.
Write $Q_1<Q_2<\cdots<Q_m$ for the \emph{distinct} values in
$\mathcal{P}$, and let $r_j$ denote the multiplicity of $Q_j$ in
$\mathcal{P}$. The adjacency spectrum of $\WG(R)$ consists of:
\begin{enumerate}[label=\textup{(\arabic*)}]
\item the eigenvalue $0$ with multiplicity $\sum_{k\in F}(|V_k|-1)$,
contributed by the independent sets $V_k$;
\item for each $j$ with $r_j\geq 2$, the eigenvalue $-Q_j$ with
multiplicity $r_j-1$, contributed by the union of the size-$Q_j$ parts;
\item $m$ further eigenvalues \textup{(}counted with multiplicity\textup{)},
the roots of the secular equation
\begin{equation}\label{eq:secular}
\sum_{j=1}^{m}\frac{r_j\,Q_j}{x+Q_j}\;=\;1.
\end{equation}
\end{enumerate}
In particular, $\WG(R)$ has at most $1+\#\{j:r_j\geq 2\}+m$ distinct
adjacency eigenvalues.
\end{theorem}

\begin{proof}
By Theorem \ref{thm:structure}, $\WG(R)$ is a complete multipartite
graph with $t=|F|+|V_0|$ parts: each part $V_k$ ($k\in F$) of size
$|V_k|$, and $|V_0|$ singleton parts. The total number of parts is
$\sum_{j=1}^m r_j=t$.

\textbf{Multiplicity of $0$.} For each $k\in F$, let $W_k$ be the
subspace of $\mathbb{R}^{\Z(R)^{\ast}}$ consisting of vectors supported in
$V_k$ with $\sum_{v\in V_k}\phi(v)=0$. For $\phi\in W_k$ and any
vertex $w$: if $w\in V_k$ then $\phi$ is supported in $V_k$ but $w$
is not adjacent to any other vertex of $V_k$, so $(A\phi)_w=0$; if
$w\notin V_k$ then $w$ is adjacent to every vertex of $V_k$, so
$(A\phi)_w=\sum_{u\in V_k}\phi(u)=0$. Thus $A\phi=0$ and
$W_k\subseteq\ker A$. The subspaces $W_k$ are pairwise orthogonal
(disjoint supports), so $\bigoplus_{k\in F}W_k\subseteq\ker A$ has
dimension $\sum_{k\in F}(|V_k|-1)$, showing that the eigenvalue $0$
has multiplicity at least $\sum_{k\in F}(|V_k|-1)$.

\textbf{Multiplicity of $-Q_j$ for $r_j\geq 2$.} Fix $j$ with
$r_j\geq 2$, and let $S_j$ denote the union of the $r_j$ parts of
$\WG(R)$ of size $Q_j$. The induced subgraph on $S_j$ is the complete
$r_j$-partite graph $K(Q_j,\ldots,Q_j)$ with $r_j$ equal parts. Let
$U_j$ be the subspace of vectors $\phi$ supported in $S_j$ that are
constant on each size-$Q_j$ part with the constants summing to zero:
$\phi|_P\equiv c_P$ for each size-$Q_j$ part $P\subseteq S_j$, and
$\sum_P c_P=0$. Then $\dim U_j=r_j-1$ (the space of constant-on-parts
vectors on $S_j$ is $r_j$-dimensional, cut to $r_j-1$ by the linear
constraint $\sum_P c_P=0$).

For $\phi\in U_j$ and any vertex $w$: if $w\in S_j$, say $w$ lies in
the size-$Q_j$ part $P$ with $\phi|_P\equiv c_P$, then $w$ is adjacent
to every vertex of $S_j\setminus P$ (and to no other vertex of $S_j$).
Hence
\[
(A\phi)_w \;=\; \sum_{P'\neq P}Q_j\,c_{P'}
            \;=\; Q_j\bigl(\textstyle\sum_{P'}c_{P'}-c_P\bigr)
            \;=\; -Q_j\,c_P \;=\; -Q_j\,\phi(w).
\]
If $w\notin S_j$ (so $w$ lies in a part of size $\neq Q_j$), then
$w$ is adjacent to every vertex of $S_j$, and
\[
(A\phi)_w \;=\; \sum_{P\subseteq S_j}Q_j\,c_P
            \;=\; Q_j\cdot 0 \;=\; 0 \;=\; -Q_j\,\phi(w),
\]
since $\phi(w)=0$. Thus $A\phi=-Q_j\,\phi$ and $U_j$ is contained in
the $(-Q_j)$-eigenspace, showing that $-Q_j$ has multiplicity at
least $r_j-1$.

\textbf{The remaining eigenvalues.} We verify orthogonality of the
subspaces $\{W_k : k\in F\}$ and $\{U_j : r_j\geq 2\}$ in three cases.
First, for $k\neq k'$ in $F$, the subspaces $W_k$ and $W_{k'}$ have
disjoint supports $V_k$ and $V_{k'}$, so $W_k\perp W_{k'}$. Second,
for $j\neq j'$ with $r_j,r_{j'}\geq 2$, the supports $S_j$ and
$S_{j'}$ are disjoint \textup{(}since $Q_j\neq Q_{j'}$ and no part of
$\WG(R)$ has two different sizes\textup{)}, so $U_j\perp U_{j'}$.
Third, for $k\in F$ and $j$ with $r_j\geq 2$, there are two sub-cases.
If $|V_k|\neq Q_j$, then $V_k\cap S_j=\varnothing$ and the supports
are disjoint. If $|V_k|=Q_j$, then $V_k\subseteq S_j$, and for
$\phi\in W_k$ and $\psi\in U_j$, $\psi$ is constant on $V_k$ \textup{(}say
$\psi\equiv c$ on $V_k$\textup{)}, giving
$\langle\phi,\psi\rangle=\sum_{v\in V_k}\phi(v)\psi(v)=
c\sum_{v\in V_k}\phi(v)=0$ by the zero-sum condition on $\phi$.

The direct sum $\bigoplus_{k\in F}W_k\oplus\bigoplus_{j:r_j\geq 2}U_j$
therefore has dimension
\[
D \;=\; \sum_{k\in F}(|V_k|-1)+\sum_{j:\,r_j\geq 2}(r_j-1)
    \;=\; (N-t)+(t-m) \;=\; N-m,
\]
using $\sum_{k\in F}(|V_k|-1)=N-t$ and
$\sum_{j:r_j\geq 2}(r_j-1)=\sum_{j=1}^m(r_j-1)=t-m$ \textup{(}the terms
with $r_j=1$ contribute zero\textup{)}. Hence the orthogonal complement
has dimension $N-D=m$, and the remaining $m$ eigenvalues are those of
the restriction of $A$ to this complement.

To identify them, apply the Esser--Harary formula
(Theorem \ref{thm:EH}): $\mu_A(\WG(R),x)=x^{N-t}B(x)$ where $B(x)$
has degree $t$. From (2), $\prod_{j=1}^{m}(x+Q_j)^{r_j-1}$ divides
$B(x)$ \textup{(}each $-Q_j$ is a root of $B(x)$ of multiplicity at
least $r_j-1$, with the convention that $r_j=1$ contributes a factor
of $1$\textup{)}, so we may write
\[
B(x) \;=\; \prod_{j=1}^{m}(x+Q_j)^{r_j-1}\cdot P(x)
\]
for some polynomial $P(x)$. Comparing degrees gives
$\deg P(x)=t-\sum_{j=1}^{m}(r_j-1)=t-(t-m)=m$. The remaining
adjacency eigenvalues are the roots of $P(x)$. To find them, group
the $q_i$ in Theorem \ref{thm:EH} by their distinct values:
$\prod_{i=1}^{t}(x+q_i)=\prod_{j=1}^{m}(x+Q_j)^{r_j}$ and
$\sum_{i=1}^{t}q_i/(x+q_i)=\sum_{j=1}^{m}r_jQ_j/(x+Q_j)$. The
Esser--Harary identity then reads
$B(x)/\prod_{j=1}^{m}(x+Q_j)^{r_j}=1-\sum_{j=1}^{m}r_jQ_j/(x+Q_j)$,
and substituting $B(x)=\prod_{j}(x+Q_j)^{r_j-1}\cdot P(x)$ yields
\[
\frac{P(x)}{\prod_{j=1}^{m}(x+Q_j)} \;=\; 1-\sum_{j=1}^{m}\frac{r_j\,Q_j}{x+Q_j}.
\]
Evaluating at $x=-Q_{j_0}$ shows $P(-Q_{j_0})=-r_{j_0}Q_{j_0}
\prod_{\ell\neq j_0}(Q_\ell-Q_{j_0})\neq 0$ \textup{(}since the $Q_\ell$
are distinct\textup{)}, so $P(x)$ has no root at any $-Q_j$.
Consequently the multiplicity of $-Q_j$ as a root of $B(x)$ is
exactly $r_j-1$, and the $m$ roots of $P(x)$ are precisely the
solutions of \eqref{eq:secular}, proving (3). The dimension count
$D+m=N$ then forces the lower bounds in (1) and (2) to be equalities.

The distinct-eigenvalue bound follows: (1) contributes the single
value $0$, (2) contributes at most $\#\{j:r_j\geq 2\}$ values, and
(3) contributes at most $m$ values.
\end{proof}

\begin{remark}
The decomposition in Theorem \ref{thm:adj-WG} unifies the
within-part and cross-part contributions. The eigenvalue $0$ comes
from vectors supported in a single part with zero sum
\textup{(}reflecting that each part is an independent set\textup{)},
while the eigenvalue $-Q_j$ for $r_j\geq 2$ comes from
constant-on-part vectors across the $r_j$ parts of size $Q_j$ with
constants summing to zero \textup{(}reflecting that the union of
size-$Q_j$ parts induces the complete $r_j$-partite graph
$K(Q_j,\ldots,Q_j)$\textup{)}. The case $Q_j=1$ recovers the
contribution of singleton parts: their union induces the clique
$K_{r_1}$, contributing $-1$ with multiplicity $r_1-1$.
\end{remark}

\begin{remark}
Theorem \ref{thm:adj-WG} parallels the Laplacian picture of Theorem
\ref{thm:main} but is genuinely different in character: the Laplacian
spectrum is integral (Corollary \ref{cor:integral}), whereas the roots
of \eqref{eq:secular} are typically irrational, so $\WG(R)$ is in
general \emph{not} adjacency integral. The integer parts of the
adjacency spectrum are $0$ and the eigenvalues $-Q_j$ from repeated
part-sizes; the secular roots \eqref{eq:secular} contribute the
non-integer eigenvalues in general.
\end{remark}

\begin{example}\label{ex:adj-Z6}
For $R=\Z_6\cong\Z_2\times\Z_3$ we have $\WG(R)=K(2,1)$, so $F=\{1,2\}$
with $|V_1|=2$, $|V_2|=1$, and $V_0=\varnothing$. The part-size
multiset is $\mathcal{P}=\{1,2\}$, so $m=2$ with $Q_1=1$, $Q_2=2$,
$r_1=r_2=1$ \textup{(}no repeated sizes\textup{)}. By Theorem
\ref{thm:adj-WG}: the eigenvalue $0$ has multiplicity
$(2-1)+(1-1)=1$; no eigenvalues from (2) since both $r_j=1$; and the
$m=2$ remaining eigenvalues solve
$\tfrac{1}{x+1}+\tfrac{2}{x+2}=1$, i.e.\ $x^{2}=2$. Hence the
adjacency spectrum is $\{\,-\sqrt2,\,0,\,\sqrt2\,\}$, confirming
non-integrality.
\end{example}

\begin{example}\label{ex:adj-Z12}
For $R=\Z_{12}\cong\Z_4\times\Z_3$ we have $F=\{2\}$, $|V_2|=2$,
$|V_0|=5$, $N=7$. The part-size multiset is
$\mathcal{P}=\{1,1,1,1,1,2\}$, so $m=2$ with $Q_1=1$, $r_1=5$, and
$Q_2=2$, $r_2=1$. By Theorem \ref{thm:adj-WG}: eigenvalue $0$ with
multiplicity $|V_2|-1=1$; from (2) the eigenvalue $-1$ with
multiplicity $r_1-1=4$; and the $m=2$ remaining eigenvalues solve
$\tfrac{5\cdot 1}{x+1}+\tfrac{1\cdot 2}{x+2}=1$, that is
$x^{2}-4x-10=0$, giving $x=2\pm\sqrt{14}$. The adjacency spectrum is
$\{\,2-\sqrt{14},\,(-1)^{(4)},\,0,\,2+\sqrt{14}\,\}$.
\end{example}

\begin{example}[A repeated non-singleton size]\label{ex:adj-Z3Z3}
For $R=\Z_3\times\Z_3$, both factors are fields and $F=\{1,2\}$ with
$|V_1|=|V_2|=|\U(\Z_3)|=2$, and $V_0=\varnothing$. Thus
$\WG(R)=K(2,2)=K_{2,2}$ on $N=4$ vertices. The part-size multiset is
$\mathcal{P}=\{2,2\}$, so $m=1$ with $Q_1=2$ and $r_1=2$. By Theorem
\ref{thm:adj-WG}: the eigenvalue $0$ has multiplicity $(2-1)+(2-1)=2$;
from (2), the eigenvalue $-Q_1=-2$ appears with multiplicity
$r_1-1=1$; and the $m=1$ remaining eigenvalue solves
$\tfrac{2\cdot 2}{x+2}=1$, giving $x=2$. The adjacency spectrum is
$\{\,-2,\,0,\,0,\,2\,\}$, matching the known spectrum of $K_{2,2}$.
\end{example}

\begin{remark}\label{rem:adj-nonintegral}
A complete characterization of the finite commutative rings $R$ for
which $\WG(R)$ is adjacency integral remains open. By Theorem
\ref{thm:adj-WG} this depends only on the part-size multiset
$\mathcal{P}$, and reduces to the question of when all roots of the
rational equation \eqref{eq:secular} are integers, an elementary
problem in the distinct values $Q_1<\cdots<Q_m$ and their
multiplicities $r_1,\ldots,r_m$.
\end{remark}

\section{Conclusion}

The complete multipartite structure (Theorem \ref{thm:structure})
underlies all of the spectral results in this paper. It also opens
several further directions.

\emph{Other spectra.} The signless Laplacian, normalized Laplacian,
and distance-Laplacian spectra of complete multipartite graphs are
known, and can be specialized to $\WG(R)$ along the same lines. For
$R=\Z_n$ the normalized Laplacian has been treated in \cite{NRA};
extending other spectral notions to arbitrary finite commutative
rings via Theorem \ref{thm:structure} is a natural next step.

\emph{Adjacency integrality.} As noted in Remark
\ref{rem:adj-nonintegral}, a complete characterization of finite
commutative rings $R$ with $\WG(R)$ adjacency integral is open.

\emph{Energy and other invariants.} The Laplacian energy, spectral
radius, and various entropy-like invariants of $\WG(R)$ can be read
off from Theorem \ref{thm:main}.

\emph{Beyond finite rings.} The structure theorem (Theorem
\ref{thm:structure}) uses only that $R$ is Artin, so it extends
verbatim to any Artin commutative ring: $R\cong\prod_{i=1}^n R_i$
with each $R_i$ an Artin local ring \textup{(}possibly
infinite\textup{)}, and $\WG(R)$ is the complete multipartite graph
with parts $\{V_k\}_{k\in F}$ together with the singleton parts from
$V_0$. For instance, for $R=\mathbb{C}\times\mathbb{C}$ \textup{(}or
$\mathbb{R}\times\mathbb{R}$\textup{)}, each factor is an Artin local
ring \textup{(}a field, in fact\textup{)}, so $F=\{1,2\}$,
$V_0=\varnothing$, and $\WG(R)$ is the complete bipartite graph
$K(|V_1|,|V_2|)$ with both parts infinite. The spectral results of
Sections \ref{sec:laplacian} and \ref{sec:adj}, however, are
statements about finite-dimensional matrices on finite graphs and do
not extend to such infinite cases without substantial reformulation;
spectral theory of the resulting infinite graphs is a separate
problem. For non-Artin commutative rings the local-product
decomposition need not exist and the partition
$\{V_k\}_{k\in F}\cup\{V_0\}$ is not available; identifying the
right analogue there is an open problem.

\end{document}